\providecommand{\pgfsyspdfmark}[3]{}

\documentclass{article}
\usepackage{mathrsfs}
\usepackage{amssymb}
\usepackage{cite}
\usepackage[all]{xy}
\usepackage{graphicx}
\usepackage{textcomp}
\usepackage{charter}
\usepackage[vflt]{floatflt}
\usepackage{mathtools}
\usepackage{enumerate}
\usepackage{wasysym}
\usepackage{rotating}
\usepackage{pdflscape}
\usepackage{fullpage}
\usepackage{authblk}
\usepackage{comment}
\usepackage{bm}
\usepackage[pdftex,colorlinks]{hyperref}
\hypersetup{
bookmarksnumbered,
pdfstartview={FitH}}

\usepackage{tikz-cd}
\usepackage{nicematrix}
\usepackage{colortbl}
\usepackage{color}

\usepackage{blkarray}
\usepackage{amsmath}
\usepackage{amsthm}

\usepackage{setspace}

\usepackage{amsfonts}
\usepackage{eucal}
\usepackage{latexsym}
\usepackage{mathdots}
\usepackage{mathrsfs}
\usepackage{enumitem}

\newtheorem{thm}{Theorem}[section]
\newtheorem{prop}[thm]{Proposition}
\newtheorem{lem}[thm]{Lemma}
\newtheorem{cor}[thm]{Corollary}
\newtheorem*{thm*}{Theorem}
\newtheorem*{cor*}{Corollary}
\newtheorem*{prop*}{Proposition}

\theoremstyle{definition}

\theoremstyle{remark}
\newtheorem{remark}[thm]{Remark}
\newtheorem{example}[thm]{Example}

\numberwithin{equation}{section}

\usepackage{mathrsfs}
\newcommand{\GL}{\mathrm{GL}}

\newcommand{\tr}{\operatorname{tr}}
\renewcommand{\Re}{\mathrm{Re}}

\newcommand{\cA}{{\mathcal A}}
\newcommand{\cB}{{\mathcal B}}

\newcommand{\cF}{{\mathcal F}}

\newcommand{\cH}{{\mathcal H}}

\newcommand{\cK}{{\mathcal K}}

\newcommand{\cO}{{\mathcal O}}

\newcommand{\cS}{{\mathcal S}}

\newcommand{\fb}{{\mathfrak b}}

\newcommand{\fp}{{\mathfrak p}}

\newcommand{\fr}{{\mathfrak r}}

\newcommand{\A}{{\mathbb A}}
\newcommand{\B}{{\mathbb B}}

\newcommand{\C}{{\mathbb C}}

\newcommand{\N}{{\mathbb N}}

\newcommand{\F}{{\mathbb F}}

\newcommand{\D}{{\mathbb D}}
\newcommand{\T}{{\mathbb T}}

\newcommand{\bH}{\mathbb{H}}

\newcommand{\be}{\begin{equation}}
\newcommand{\ee}{\end{equation}}
\def\ba{\begin{eqnarray*}}
\def\ea{\end{eqnarray*}}

\newcommand{\bi}{\begin{itemize}}
\newcommand{\ei}{\end{itemize}}
\newcommand{\bn}{\begin{enumerate}}
\newcommand{\en}{\end{enumerate}}
\newcommand{\bbm}{\begin{bmatrix}}
\newcommand{\ebm}{\end{bmatrix}}
\newcommand{\bpm}{\begin{pmatrix}}
\newcommand{\epm}{\end{pmatrix}}
\newcommand{\bsm}{\left ( \begin{smallmatrix}}
\newcommand{\esm}{\end{smallmatrix} \right) }

\newcommand{\mr}{\ensuremath{\mathrm}}
\newcommand{\scr}{\ensuremath{\mathscr}}

\newcommand{\mf}{\ensuremath{\mathfrak}}
\newcommand{\ov}{\ensuremath{\overline}}
\newcommand{\sm}{\ensuremath{\setminus}}
\newcommand{\wt}{\ensuremath{\widetilde}}

\newcommand{\la}{\ensuremath{\lambda }}
\newcommand{\om}{\ensuremath{\omega}}

\newcommand{\cuntz}{\ensuremath{\cO_d}}

\def\C{\mathbb{C}}

\def\D{\mathbb{D}}

\def\N{\mathbb{N}}
\def\B{\mathbb{B}}

\def\fr{\mathfrak{r}}
\def\frt{\mathfrak{r} ^\mathrm{t}}
\def\fb{\mathfrak{b}}
\def\fbt{\mathfrak{b} ^{\mathrm{t}}}

\def\bH{\mathbb{H}}

\def\A{\mathbb{A} _d}

\def\fp{\mathbb{C} \langle \mathbb{\mathfrak{z}} \rangle }

\def\mrt{\mathrm{t}}
\def\posncm{(\mathscr{A} _d ^\dagger) _+}
\def\hardy{\mathbb{H} ^2 _d}
\def\mult{\mathbb{H} ^\infty _d}
\newcommand{\ip}[2]{\ensuremath{\langle {#1} , {#2} \rangle}}
\def\nbdom{\mr{Dom} \, }

\def\fskew{\C \ \mathclap{\, <}{\left( \right.} \,  \mf{z} \, \mathclap{ \, \, \, \, >}{\left. \right)} }
\def\ratfps{\C _0 \ \mathclap{\, <}{\left( \right.} \,  \mf{z} \, \mathclap{ \, \, \, \, >}{\left. \right)} }

\title{Rational Cuntz states peak on the free disk algebra}

\author[1]{Robert T.W. Martin\thanks{Supported by NSERC grant 2020-05683}}
\affil[1]{\footnotesize University of Manitoba}

\author[2]{Eli Shamovich}
\affil[2]{ \footnotesize Ben-Gurion University of the Negev}

\date{}

\begin{document}
\maketitle

\begin{abstract}
We apply realization theory of non-commutative rational multipliers of the Fock space, or \emph{free Hardy space} of square--summable power series in several non-commuting variables to the convex analysis of states on the Cuntz algebra. We show, in particular, that a large class of Cuntz states which arise as the `non-commutative Clark measures' of isometric NC rational multipliers are peak states for Popescu's free disk algebra in the sense of Clou\^atre and Thompson.
\end{abstract}

\section{Introduction}

This paper applies non-commutative (NC) analysis to a question in non-commutative convexity. Our main result constructs states on the Cuntz algebra, $\cuntz$, which peak at non-commutative rational inner functions in Popescu's \emph{free disk algebra}, $\A$.  We also provide a novel characterization of unital quantum channels in terms of NC rational inner functions.

Non-commutative convexity first appeared in the seminal work of Arveson \cite{Arv-subalg}. In \cite{Arv-subalg}, Arveson extended classical Choquet theory to a non-commutative, operator--algebraic setting. Classical Choquet theory studies the extreme boundary of compact convex sets and representing measures. Namely, let $K$ be a compact convex set; a probability measure $\mu$ on $K$ represents a point $x \in K$ if the restriction of the corresponding state from $C(K)$, the continuous functions on $K$, to the space of continuous affine functions on $K$, is the functional of evaluation at $x$. Bauer characterized the extreme points of $K$ as those that admit only one representing measure (necessarily $\delta_x$). This characterization lends itself to an extension to a more general setting of unital subspaces $1 \in M \subset C(X)$, where $X$ is a compact Hausdorff space. The Choquet boundary of such a subspace is the collection of all points $x \in X$, such that $\delta_x|_M$ admits a unique extension to $C(X)$. One way to obtain points in the Choquet boundary is to find peak points. A point $x \in X$ is an $M-$peak point, if there exists an $f \in M$, such that $|f(x)| = \|f\|$ and $|f(y)| < \| f \|$ for any $y \neq x$. In particular, if $1 \in A \subset C(X)$ is a uniform algebra, and $X$ is metrizable, then an important result of Bishop states that the Choquet boundary of $A$ is precisely the set of peak points of $A$.

The field of non-commutative convexity has expanded quickly with contributions from Wittstock, Effros and his collaborators, and many others. The reader is referred to the monograph of Davidson and Kennedy \cite{DavidsonKennedy-big} and the references therein for further details on non-commutative convexity theory. Of particular interest is the non-commutative Choquet theory first introduced and developed by Arveson \cite{Arv-subalg, Arv-choq1, Arveson-choq2,Arveson-choq3}. Suppose $B$ is a unital $C^*$-algebra and $1 \in A \subset B$ is an operator algebra. In this case, we say that an irreducible representation $\pi \colon B \to B(\cH)$ is a boundary representation if $\pi|_A$ has a unique extension to $B$. The image of the direct sum of all boundary representations gives the \emph{$C^*-$envelope}, $C^* _{min} (A)$ of $A$. This is a $C^*-$algebra in which $A$ embeds, completely isometrically, and it is universal and minimal in the sense that if $A$ embeds completely isometrically into any $C^*-$algebra, $B$, then there is a $*-$homomorphism from $B$ onto $C^* _{min} (A)$ which intertwines the embeddings. Existence and construction of the $C^*-$envelope via boundary representations was a long-unsolved problem in operator algebra theory until it was resolved in full generality by Davidson and Kennedy \cite{DavidsonKennedy,Arv-choq1,Arv-subalg,Hamana,DMc-boundary}. Arveson also introduced the concept of a peaking representation in \cite{Arveson-choq3}. His ideas were further extended by Clou\^{a}tre \cite{Clouatre-peaking}, Clou\^{a}tre and Thompson \cite{ClouThom-fin_dim,ClouThom-min_bound}, and Davidson and Passer \cite{DavidsonPasser}. In particular, Clou\^{a}tre and Thompson propose to study `peaking states'. Precise definitions and details on peaking states can be found in Subsection \ref{subsec:peak}. Clou\^{a}tre and Thompson \cite{ClouThom-min_bound} suggest that studying peaking phenomena on the Cuntz algebra will be interesting. We will provide a family of examples of peaking states and corresponding representations in this setting.

In more detail, we will study states on the Cuntz algebra $\cuntz$, which peak at elements of Popescu's free or non-commutative disk algebra, $\A$. Here, $\scr{E} _d := C^* \{ I , L_1 , \cdots , L_d \}$, is the \emph{Cuntz--Toeplitz algebra}, the unital $C^*-$algebra generated by the left creation operators, $L_k$, on the full Fock space, $\hardy$. Here, $\hardy$ can be defined as the Hilbert space of complex square--summable power series in several NC variables, equipped with the $\ell ^2-$inner product of the complex power series coefficients. In this viewpoint the $L_k = M ^L _{\mf{z} _k}$ act as isometric left multiplications by the $d$ independent formal NC variables, $\mf{z} = (\mf{z} _1 , \cdots , \mf{z} _d)$.  This algebra contains the compact operators, $\scr{K} (\hardy )$, on $\hardy$ and the Cuntz algebra is then defined as the quotient $C^*-$algebra,  $\cuntz := \scr{E} _d / \scr{K} ( \hardy )$. Popescu's free disk algebra $\A := \mr{Alg} \{ I , L_1, \cdots L_d \}$ is the unital norm-closed algebra generated by the left creation operators. Moreover, $\A$ can be identified, completely isometrically with $\mr{Alg} \{ I, S_1, \cdots , S_d \}$, where the $S_j$ denote the generators of the Cuntz algebra.

Our primary tools are free analysis and non-commutative function theory. These are rapidly growing fields in modern analysis. Free or non-commutative analysis was initially motivated by the study of analytic functional calculus of several commuting and non-commuting operators, as pioneered by Taylor \cite{Taylor-frame,Taylor-ncfunc}, Voiculescu's free probability theory \cite{Voiculescu-quest1,Voiculescu-quest2}, and Takesaki's extension of Gelfand duality to arbitrary, non-commutative $C^*-$algebras in \cite{Takesaki}. Popescu \cite{Pop-dil,Pop-vN,Pop-freeholo,Pop-freeholo2} studied non-commutative functions to extend the classical Sz. Nagy--Foias theory of dilations and von Neumann's inequality to the multivariable setting. He introduced the full Fock space as an analog of the classical Hardy Hilbert space of analytic function in the complex unit disk, $\D$. We provide the necessary definitions and background theory in Subsection \ref{subsec:hardy}.

A particularly well-studied class of NC functions is the set of all non-commutative rational functions. These functions arise naturally in many different branches of pure and applied mathematics, such as the theory of localizations and quasideterminants in non-commutative rings \cite{Amitsur,Cohn,GGRW}, the theory of formal languages \cite{BR,HW}, and systems theory \cite{BGM1,BKV}. This paper will focus on non-commutative rational functions that are elements of the weak operator topology (WOT)-closed algebra generated by the left creation operators on the full Fock space. Such functions were extensively studied in \cite{JMS-NCrat,JMS-ncratClark}. In this paper, there is a critical interplay between non-commutative rational inner functions, \emph{i.e.} NC rational functions which define isometric (inner) multipliers on the Fock space, states on the Cuntz algebra which peak on the free disk algebra, and finite-dimensional row coisometries. A row contraction is a contractive linear map from several copies of a Hilbert space into one copy. The necessary background and details are in Subsection \ref{subsec:ncrat}.

Given an irreducible and finite-dimensional row coisometry, $T = (T_1 , \cdots , T_d )$, and a unit vector $x$, we can define a linear functional on $\A$ via $\mu(f) = \langle x, f(T) x \rangle$. One can apply a Gelfand--Naimark--Segal construction to $(\mu , \A )$, and this yields a GNS-Hilbert space, $\hardy (\mu )$ and a $*-$representation of the Cuntz--Toeplitz $C^*-$algebra, $\pi _\mu$. Since $T$ is a row-coisometry, one can show that $\Pi _\mu := \pi _\mu (L)$ is a Cuntz, \emph{i.e.} surjective row isometry, so that $\mu$ admits a unique extension to a state on $\cO_d$, $\hat{\mu}$ by \cite[Proposition 5.11]{JMT-ncFnM}. This state, $\hat{\mu}$, is a finitely--correlated state, as introduced by Bratteli and J\o rgensen \cite{BraJorg}. Our main result is Theorem \ref{thm:main}, which states that every such finitely--correlated state is an $\A$-peak state. In particular, by results of Clou\^atre, this implies that every such state is an exposed extreme point of the state space of the non-commutative or free disk operator system. The proof of the theorem constructs a non-commutative rational inner, such that the state peaks at it. However, this inner is constructed from $(T^t _1 , \cdots, T^\mrt _d )$, the coordinate-wise transpose of $T$. This duality leads us to Theorem \ref{thm:qc}, which shows that a non-commutative rational inner $\fb$ arises from a quantum channel if and only if the `transpose' $\fb^t$, obtained by reversing the order of all products in monomials in the power series of $\fb$, is an inner as well. En route to these results, we obtain some results on spectra of non-commutative rational functions regular at the origin, which refine and extend our previous results in \cite{JMS-ncratClark}.

\section{Background} \label{sec:background}

\subsection{Non-commutative Hardy space and its multipliers} \label{subsec:hardy}

Given $d \in \N$, the \emph{full Fock space} is the Hilbert space direct sum $\bH^2_d = \oplus_{n=0}^{\infty} \left(\C^d\right)^{\otimes n}$, with the usual convention of $\left(\C^d\right)^{\otimes 0} \cong \C$. We will call $\hardy$ the \emph{free} or \emph{non-commutative} (NC) \emph{Hardy space} as it has many properties in common with the classical Hardy space, $H^2$, of square--summable power series in the complex unit disk. In particular, the elements of $\bH^2_d$ can be viewed as NC functions on the NC unit ball
\[
\B ^d _{\N} = \bigsqcup_{n=1}^{\infty} \B _n ^d, \text{ where } \B ^d _n  = \left\{ X \in \scr{B} (\C^n \otimes \C^d ,\C^n) \mid X X^* < I\right\}.
\]
The interpretation of $\hardy$ as a space of non-commutative functions was first considered by Popescu in \cite{Pop-freeholo}. A natural way to see $\bH^2_d$ as a space of functions is by noting that $\bH^2_d$ is the completion of the free algebra, $\C\langle \mathfrak{z}_1,\ldots,\mathfrak{z}_d \rangle = \fp$, of NC or free polynomials, with respect to the inner product that makes the monomials orthonormal. The NC monomials correspond to words in the alphabet $\{1,\ldots,d\}$. Given a word $\alpha = i_1 \cdots i_n$, $i_j \in \{ 1, \cdots , d \}$, we write $\mathfrak{z}^{\alpha} = \mathfrak{z}_{i_1} \cdots \mathfrak{z}_{ i_n}$. If the word is empty, we define $\mf{z} ^\emptyset := 1$. We will also denote the length of $\alpha = i_1 \cdots i_n$ by $|\alpha | =n$. For a $d$-tuple of operators $T = (T_1,\ldots,T_d)$, we set $T^{\alpha} = T_{i_1} \cdots T_{i_n}$ and $T^\emptyset =I$. Note that for any free polynomial, $p \in \fp$, we can evaluate $p$ on any $d$-tuple of matrices $Z = (Z_1, \cdots , Z_d ) \in \C ^{n\times n} \otimes \C ^{1\times d} =: \C ^d _n$ and obtain a function with the following properties: $p(Z) = p (Z_1, \cdots, Z_d)$  
\begin{itemize}
    \item[(i)] is \textbf{graded}: for every $X \in \C ^d _n$, $p(X) \in \C ^{n \times n}$,

    \item[(ii)] \textbf{respects direct sums}: for every $X \in \C ^d _n$ and $Y \in \C ^d _m$, we have $$p(X \oplus Y) = p\left( \begin{pmatrix} X & 0 \\ 0 & Y \end{pmatrix} \right) = \begin{pmatrix} p(X) & 0 \\ 0 & p(Y) \end{pmatrix} = p(X) \oplus p(Y),$$

    \item[(iii)] \textbf{respects similarities}: for every $X \in \C ^d _n $ and $S \in \GL_n$, we have $$p(S^{-1} X S) = p(S^{-1} X_1 S, \ldots, S^{-1} X_d S) = S^{-1} p(X) S.$$
\end{itemize}
The elements of $\bH^2_d$, thus, can be viewed as power series in non-commuting variables. The power series converge uniformly and absolutely on all tuples of matrices of norm $\leq r$, for every $0 < r < 1$. The properties above hold, except that the third property needs to be modified to hold only if $S^{-1} X S \in \B_\N ^d$. Therefore, we obtain a space of NC functions on $\B_\N ^d$.

In direct analogy with classical Hardy space theory, (left) multiplication by any of the $d$ independent variables define isometries, $L_j = M^L _{\mf{z} _j}$ on $\hardy$ with pairwise orthogonal ranges. The `row' operator, $$L = (L_1,\ldots,L_d) \colon \bH^2_d \otimes \C^d \to \bH^2_d,$$ is then a \emph{row isometry}, \emph{i.e.} an isometry from several copies of a Hilbert space into one copy. The operators $L_j$ are called the left creation operators or the left free shifts. The WOT-closed algebra $\bH^{\infty}_d$, the \emph{free Hardy algebra}, generated by the $L_j$ was extensively studied by Popescu \cite{Pop-freeholo,Pop-freeholo2,Pop-vN} and Davidson--Pitts \cite{DP-inv,DP-alg,DP-NevPick}. The free Hardy algebra is completely isometrically isomorphic to the algebra of all uniformly bounded NC functions on $\B_{\N} ^d$ with the uniform norm.
\[
\|f\| = \sup_{X \in \B_{\N} ^d}\|f(X)\|.
\]
The space $\bH^2_d$ is an NC reproducing kernel Hilbert space (RKHS) in the sense of Ball, Marx, and Vinnikov \cite{BMV}. The free Hardy algebra can then be identified as the algebra of left multipliers of this NC-RKHS.

Popescu and Davidson--Pitts have generalized the classical inner-outer factorization of bounded analytic functions on the disk to $\bH^{\infty}_d$. Here, as in classical $H^2$ theory, an inner NC function in $\bH^{\infty}_d$ defines an isometric left multiplier on $\bH^2_d$ and an outer NC function is one that defines a left multiplier with a dense range.

Recall that the free disk algebra, $\A$, is the unital norm-closed algebra generated by the left creation operators on the full Fock space. We assume throughout that $d \geq 2$. This algebra is a very close NC analogue of the classical disk algebra, $\mathbb{A}_1 = A(\D)$. The Cuntz algebra is the universal $C^*$-algebra of a surjective row isometry. That is, if  $S = (S_1,\ldots,S_d)$ denotes the row isometry of generators of the Cuntz algebra, then $\sum_{j=1}^d S_j S_j^* = I$. The quotient map from the Cuntz--Toeplitz algebra, $\scr{E} _d$ onto $\cuntz$ restricts to a complete isometry on $\A$, and since the Cuntz algebra is simple, it is the $C^*$-envelope of $\A$. Classically, $C(\T)$ is the $C^*$-envelope of $A(\D)$. Moreover, $C(\T) = \overline{A(\D) + A(\D)^*}$. In the NC setting, $\scr{A} _d := ( \A + \A ^* ) ^{-\| \cdot \|}$ is an operator system that embeds, completely isometrically, into $\cuntz$, but does not coincide with the Cuntz algebra. In particular, $\scr{A} _d$ is not an algebra or a $C^*-$algebra. Namely, by \cite[Theorem 3.1]{Pop-ncdisk}, if $S_1, \cdots, S_d$ denote the generators of $\cuntz$, then the free disk algebra $\A := \mr{Alg} \{ I , L_1 , \cdots , L_d \} ^{- \| \cdot \|}$ and $\A (S) := \mr{Alg} \{ I , S_1 , \cdots , S_d \} ^{- \| \cdot \|}$ are completely isometrically isomorphic. Moreover, by \cite[Proposition 3.5]{Paulsen-cbmaps}, the \emph{free disk system}, $\scr{A} _d := ( \A + \A ^* ) ^{-\| \cdot \|}$ and the operator system $\scr{A} _d (S) = ( \A (S) + \A (S) ^* ) ^{-\| \cdot \|}$ are then completely isometrically isomorphic. For the remainder of the paper, we identify $\A$ with $\A (S)$ and $\scr{A} _d$ with $\scr{A} _d (S)$ so that the free disk algebra and the free disk system are viewed as subspaces of the Cuntz algebra, $\cuntz$. An important property of $\A$ is that it is semi-Dirichlet, namely that $\A^* \A \subset \scr{A}_d$. This property enables one to perform a Gelfand--Naimark--Segal (GNS)-type construction directly from positive linear functionals on $\scr{A}_d$. Jury and the first author \cite{JM-freeAC} have developed an NC extension of the classical Alexandrov--Clark measure theory. In this NC theory, the positive linear functionals on the free disk system, denoted by $\posncm$, play the role of positive measures on the unit circle. There is, in particular, a one-to-one correspondence between states on $\scr{A}_d$ and contractive functions in $\bH^{\infty}_d$ that vanish at $0$. We will be primarily interested in the positive \emph{finitely--correlated states} or `NC measures' which arise as the `NC Clark measures' of NC rational multipliers of the Fock space, and we will introduce these in more detail in the following section.

\subsection{Non-commutative rational functions} \label{subsec:ncrat}

The theory of non-commutative rational functions has been developed independently in pure and applied disciplines ranging from pure algebra to computational and systems theory. In particular, the algebra of all NC rational functions is the \emph{free skew field} as constructed by Amitsur \cite{Amitsur} and Cohn \cite{Cohn} and is denoted by $\fskew$. (In NC algebra, Amitsur and Cohn proved that this `free skew field' is the universal `field of fractions' of the free algebra, $\fp$, of free or non-commutative complex polynomials in the $d$ NC variables, $\mf{z} = ( \mf{z} _1, \cdots , \mf{z} _d )$.) The domain of an NC rational function is roughly the largest collection of $d$-tuples of matrices to which our NC rational function can be continued. We will denote the domain of $\fr$ by $\nbdom \fr$. This paper will focus on NC rational functions that are defined and bounded on $\B ^d _{\N}$. This assumption simplifies much of the theory. The interested reader should consult \cite{KVV-ncrat,KVV-ncratdiff,Volcic} and the references therein for more detail on the theory of NC rational functions.  

Since $0 \in \B ^d _\N$, we will assume that our NC rational functions always have $0$ in their domains, and we will denote the algebra of all such NC rational functions by $\ratfps$.  This assumption simplifies the definition of an NC rational function somewhat. We will say that an NC rational function in $\ratfps$ is any expression of the form
\[
\fr(\mathfrak{z}) = c^* \left(I - \sum_{j=1}^d \mathfrak{z}_j A_j\right)^{-1} b.
\]
Here $A_1,\ldots,A_d \in \C ^d _n$ and $b,c \in \C^n$. The evaluation of such a function on a tuple $Z_1,\ldots, Z_d \in \C ^d _m$ is performed via tensor products, i.e.,
\be
\fr(Z) = (I_m \otimes c)^* \left(I_m \otimes I_n - \sum_{j=1}^d Z_j \otimes A_j\right)^{-1} (I_m \otimes b). \label{realize}
\ee
One should note that this is not the original definition of an NC rational function but rather a result in the sense that an NC rational function is usually defined as a certain equivalence class of valid `NC rational expressions' obtained by applying the arithmetic operations `$+, \cdot$, and, $^{-1}$', to the free algebra, $\fp$.  One can then prove that any such NC rational function in $\ratfps$ obeys a `realization formula' as in Equation (\ref{realize}) above. Such a triple, $(A,b,c) \in \C ^d _n \times \C ^n \times \C ^n$ is called a descriptor realization of the NC rational function $\fr$. For every NC rational function, there exist many such descriptor realizations. However, there exists one with $n$ minimal, the \emph{minimal realization}. The article ``the'' is justified because two realizations of $r$ with minimal $n$ are jointly similar. Namely, given two minimal realizations $(A,b,c)$ and $(\wt{A}, \wt{b}, \wt{c})$, if
\[
\tilde{c}^* \left(I - \sum_{j=1}^d \mathfrak{z}_j \tilde{A}_j\right)^{-1} \tilde{b} = \fr(\mathfrak{z}) = c^* \left(I - \sum_{j=1}^d \mathfrak{z}_j A_j\right)^{-1} b,
\]
then there exists $S \in \GL_n$, such that $\tilde{A}_j = S^{-1} A S$, for $1 \leq j \leq d$, $\tilde{b} = S b$, and $\tilde{c} = S^{-1 *} c$ \cite[Theorem 2.4]{BR}. Moreover, a realization $(A,b,c)$ of size $n$ is minimal if and only if $c$ is $A-$cyclic and $b$ is $A^*-$cyclic in the sense that 
$$ \bigvee _{\om \in \F ^d} A^\om c = \C ^n = \bigvee A^{*\om} b, $$ where $\F ^d$ denotes the free monoid of all words in the $d$ letters $\{ 1 , \cdots , d \}$. We will write $L_A(\mathfrak{z}) := I - \sum_{j=1}^d \mathfrak{z}_j \tilde{A}_j$ whose inverse appears in the above expression. This object, $L_A (Z)$, is called a \emph{linear pencil} (it is affine linear). It is a result of Vinnikov--Kaliuzhnyi-Verbovetskyi \cite{KVV-ncrat} and Vol\v{c}i\v{c} \cite{Volcic} that the domain of our function $\fr$ can be described as the collection of all $Z$, such that $\det L_A(Z) \neq 0$, where $A$ comes from the minimal realization.

NC rational functions in $\bH^2_d$ and $\bH^{\infty}_d$ were studied by Jury, and the authors in \cite{JMS-NCrat}. In fact, $\fr \in \bH^2_d$ if and only if the tuple $A = (A_1,\ldots,A_d)$ appearing in its minimal realization has joint spectral radius strictly less than $1$. Here, the joint spectral radius of $A$ was defined by Popescu \cite{PopRota} as a natural multivariate analogue of Beurling's spectral radius formula,
\[
\rho(A) = \lim_{n\to\infty} \left\|\sum_{|\alpha| = n} A^{\alpha} A^{\alpha *}\right\|^{\dfrac{1}{2n}}.
\]
It further follows, by Popescu's multi-variable Rota--Strang theorem, that $\fr \in \bH^2_d$ if and only if it has a row ball of radius strictly greater than $1$ in its domain \cite{PopRota}. In particular, if this is the case, then $\fr \in \A \subset \mult$. See \cite[Theorem A]{JMS-NCrat} for several characterizations equivalent to membership of an NC rational function in the full Fock space.

Every NC rational contractive function can be associated (essentially) uniquely to an NC Clark measure, i.e., a positive functional on $\scr{A}_d$. Jury and the authors in \cite{JMS-ncratClark} characterized such linear functionals or 'NC rational Clark measures'. It turns out that the NC Clark measures that arise from inner NC rational functions in $\mult$ with $\fr(0) = 0$ are precisely the finitely--correlated states studied by Bratteli and J\o rgensen \cite{BraJorg} and later by Davidson, Kribs, and Shpigel \cite{DKS-finrow}. Here, note that an NC Clark measure, $\mu _b$, corresponding to any $b \in [ \mult ] _1$ can be any positive linear functional on the free disk system, and this functional will be a state, \emph{i.e.} $\mu _b (I ) =1$, if and only if $b(0) =0$. Moreover, \cite[Theorem 4.1]{JMS-ncratClark} provides a complete description of all NC rational inners in terms of finite--dimensional row coisometries. Since we will apply this description, we will recall it: Let $\fb$ be a contractive NC rational function in the unit row-ball with $\fb (0) = 0$. (Again, $\fb (0) =0$ ensures that its NC Clark measure, $\mu _\fb$, is a state.) Then, such an $\fb$ is inner if and only if there exists a row coisometry $T = (T_1,\cdots,T_d) \colon \C^n \otimes \C^d \to \C^n$ and a unit vector $x \in \C^n$, such that $x$ is $T$ and $T^*-$cyclic and $\fb$ is given by the realization formula,  
\be
\fb(\mathfrak{z}) = (P_0 x)^* \left( I - \sum_{j=1}^d \mathfrak{z}_j T_{0,j}^* \right)^{-1} \left( \sum_{j=1}^d z_j T_j^*x \right). \label{ncratinner}
\ee
Moreover, in this case, $\mu _\fb (L^\om ) = \mu _{T,x} (L ^\om ) := x^* T^{*\om} x$.
Here we set $\cH_0 = \bigvee_{\alpha \neq \emptyset} T^{* \alpha} x$, $P_0$ is the orthogonal projection on $\cH_0$, and for $1 \leq j \leq d$, $T_{0,j}^* = T_j^*(I - x x^*)|_{\cH_0}$. In particular, if $T$ is irreducible (the co-ordinate matrices of $T$ generate $\C ^{n\times n}$ as an algebra), then every unit vector will give rise to such a realization and $\cH_0 = \C^n$. This form of realization is called a Fornasini--Marchesini (FM) realization. More generally, an FM realization is one of the form 
\[
\fr(\mathfrak{z}) = D + C^*\left(I - \sum_{j=1}^d \mathfrak{z}_j A_j \right)^{-1} \left(\sum_{j=1}^d \mathfrak{z}_j B_j\right).
\]
Note that $D = \fr(0)$, so that if we assume that $\fr(0) = 0$, we obtain the preceding form of Equation (\ref{ncratinner}). We will denote the FM realization as $(A,B,C,D)$. One can pass from a descriptor realization to an FM one quite easily. The following lemma is well-known to experts, but since we do not have a reference, we chose to include it.

\begin{lem}
Let $\fr \in \ratfps$. If $(A , B , C , D )$ is a minimal FM realization of $\fr$, then 
$$\nbdom \fr = \left\{ Z \in  \C ^d _\N | \ \mr{det} \, L_A (Z) \neq 0 \right\}.$$
\end{lem}

\begin{proof}
By \cite[Theorem 3.5]{Volcic}, the domain of $\fr$ is the complement of the singularity locus of the linear pencil $L_{\hat{A}} (Z)$, where $(\hat{A}, b, c)$ is a minimal descriptor realization of $\fr$. A minimal FM realization, $(A' , B ' , C ' , D')$ of $\fr$ can then be constructed by setting $\cH ' := \bigvee _{\om \neq \emptyset} \hat{A}^{\om} c$, $A' := \hat{A} | _{\cH '}$, $B' := \hat{A} c$, $C' = (P_{\cH '} b) ^*$ and $D' := \fr (0)$. By uniqueness of minimal realizations (uniqueness also holds for minimal FM realizations), we can assume that $(A',B' , C' , D' ) = (A,B,C,D)$. Since $\cH '$ has codimension at most $1$, if $\cH ' \subsetneq \cH$, then $\hat{A}$ decomposes as 
$$ \hat{A} = \bpm A & * \\ 0 & a \epm, $$ where $a \in \C ^d$. Again, by \cite[Theorem 3.5]{Volcic}, if $Z \in \nbdom \fr$ then 
$$ 0 \neq \mr{det} \, L_{\hat{A}} (Z) = \mr{det} (L_{A} (Z) ) \mr{det} \, L_a (Z), $$ so that $\mr{det} \, L_A (Z) \neq 0$. Conversely if $\mr{det} \, L_A (Z)$ is not $0$ for some $Z \in \C ^d _n$, then $\fr (Z)$ is well-defined as the transfer function,
$$ D I_n + I_n \otimes C L_A (Z) ^{-1} Z \otimes B, $$ so that $Z \in \nbdom \fr$.
\end{proof}

If $\fr \in \ratfps$ with minimal descriptor realization $(A,b,c)$, we will allow the domain of $\fr$ to include $d-$tuples of operators in an infinite dimensional Hilbert space. We will denote such operator $d-$tuples by $\C ^d _\infty$ or $\scr{B} (\cH ) ^d$, where $\cH$ is a separable Hilbert space. Namely, given $Z = (Z_1 , \cdots , Z_d ) \in \C ^d _\infty$, we will say that $Z \in \nbdom \fr$ if $L_A (Z)$ is invertible.  It will be convenient to introduce some basic notations; we view any $Z \in \C ^d _n$, $n \in \N $ as a row $d-$tuple of operators, $Z = (Z_1 , \cdots , Z_d )$, $Z_j \in \C ^{n\times n}$. Any such row defines a bounded linear map from $\C ^n \otimes \C ^d$ into $\C ^n$. Hence, by $Z^*$, we mean the `column operator', $Z^* := \bsm Z_1 ^* \\ \vdots \\ Z_d ^* \esm : \C ^n \rightarrow \C ^n \otimes \C ^d$, obtained as the Hilbert space adjoint of the linear map $Z$. Similarly, $Z^{\mrt} := \bsm Z_1 ^\mrt \\ \vdots \\ Z_d ^\mrt \esm$ denotes the transpose of the row operator, $Z$, with respect to the standard bases of $\C ^n$ and $\C ^n \otimes \C ^d$. We will, however, also have occasion to consider the row operator $\mr{row} (Z^* ) := (Z_1 ^* , \cdots , Z_d ^* ) : \C ^n \otimes \C ^d \rightarrow \C ^n$ obtained as the component--wise adjoint of $Z$. The row operator $\mr{row} (Z^\mrt )$ is defined similarly. We will also consider the operation $\ov{( \cdot )} : \C ^n \rightarrow \C ^n$, defined by $x \, \mapsto \, \ov{x}$, where $\ov{x}$ denotes entry--wise complex conjugation with respect to the standard basis of $\C ^n$. If $Z = (Z_1 , \cdots , Z_d ) \in \C ^d _n$, we define $\ov{Z} := (\ov{Z} _1 , \cdots , \ov{Z} _d )$, where $\ov{Z} _j := \ov{(\cdot)} \circ Z_j \circ \ov{(\cdot)}$, so that $\ov{Z} _j$ is obtained by entry--wise complex conjugation of the matrix $Z_j$ and $\ov{Z} = (Z^* ) ^\mrt = (Z^\mrt) ^*$ is a row $d-$tuple of matrices.

The following lemmas give us a useful condition on the spectra of NC rational functions with the origin in their domains.

\begin{lem}
Let $\cH$ and $\cK$ be Hilbert spaces. Let $A \in B(\cH)$ and $D \in B(\cK)$ be invertible and $B \in B(\cK,\cH)$ and $C \in B(\cH, \cK)$. Then, the Schur complement, $A - B D^{-1} C$, has a non-trivial kernel if and only if $D - C A^{-1} B$ has a non-trivial kernel. Moreover, the map, $D^{-1} C|_{\ker (A - B D^{-1} C)} \colon \ker (A - B D^{-1} C) \to \ker (D - C A^{-1} B)$, is an isomorphism.
\end{lem}
\begin{proof}
Since the claim is symmetric, we will prove only the forward implication. Let us assume that $0 \neq v \in \ker(A - B D^{-1} C)$. Then, $0 
\neq A v = B D^{-1} C v$. In particular, $D^{-1} C v \neq 0$. Therefore,
\[
(D - C A^{-1} B) D^{-1} C v = C v - C A^{-1} B D^{-1} C v = C A^{-1}(A v- B D^{-1} C v) = 0. 
\]
To prove the last part of the claim, we let $\Phi = D^{-1} C|_{\ker (A - B D^{-1} C)}$. As we saw above, $\Phi$ is well-defined and injective. Similarly, let $\Psi = A^{-1} B|_{\ker (D - C A^{-1} B)}$. Then, for every $v \in \ker (A - B D^{-1} C)$, we have that
\[
\Psi \Phi v = A^{-1} B D^{-1} C v = A^{-1} A v = v.
\]
Similarly, in the other direction. Hence, we obtain our isomorphism.
\end{proof}

\begin{prop} \label{lem:spectrum_of_schur_complement}
Let $\fr \in \ratfps$ with minimal FM realization $(A,B,C,D)$. If $Z \in \nbdom \fr$ then $\la \neq \fr (0)$ is an eigenvalue of $\fr (Z)$ if and only if $1$ is an eigenvalue of $Z \otimes A ^{(\la)}$, where 
$$ A^{\la } _j :=   A_j  + (\la - \fr (0) ) ^{-1} B_j C.$$  If $v$ is an eigenvector of $Z \otimes A^{(\la)}$ corresponding to the eigenvalue $1$, then $I \otimes C v$ is an eigenvector of $\fr (Z)$ with corresponding eigenvalue $1$. 
\end{prop}
\begin{proof}
Consider the following matrix
\[
\begin{pmatrix} (\lambda - \fr(0)) I & I \otimes C^* \\ \sum_{j=1}^d Z_j \otimes B_j & I - \sum_{j=1}^d Z_j \otimes A_j \end{pmatrix}.
\]
Note that one Schur complement is $$\lambda - \fr(0) I - (I \otimes C^*) \left(I - \sum_{j=1}^d Z_j \otimes A_j \right)^{-1} \left(\sum_{j=1}^d Z_j \otimes A_j\right) = \lambda - \fr(Z). $$ The other is $$I - \sum_{j=1}^d Z_j \otimes A_j - \frac{1}{\lambda - \fr(0)} \sum_{j=1}^d Z_j \otimes B_j C^* = I - \sum_{j=1}^d Z_j \otimes A^{(\lambda)}_j.$$ The claim now follows from the preceding lemma.
\end{proof}

The previous proposition extends and refines \cite[Proposition 5.5]{JMS-ncratClark}. In \cite[Proposition 5.6]{JMS-ncratClark}, we applied these spectral results to prove that any NC rational inner, $\fb$, has eigenvalues of modulus $1$ when evaluated at certain points on the boundary of the unit row-ball, $\B ^d _\N$. Here, recall that NC rational inner multipliers of the Fock space, $\fb$, are in bijective correspondence with pairs $(T,x)$, where $T \in \C ^d _n$ is a row coisometry and $x$ is any vector which is both $T$ and $T^*$ cyclic. In particular, if $\fb = \fb _{T,x}$ is the unique NC rational inner corresponding to such a pair, then for any $\zeta \in \partial \D$, the NC rational inner $\ov{\zeta} \fb$ corresponds to a pair $(T(\zeta ), x)$, where $T(\zeta)$ is a rank--one row coisometric perturbation of $T$, see Equation (\ref{ncratinner}) and \cite[Theorem 4.1]{JMS-ncratClark}. Proposition 5.6 of \cite{JMS-ncratClark} then states:

\begin{prop*}
Given $\zeta \in \partial \D$, let $A_\zeta ^*$ be a column--isometric restriction of $T(\zeta ) ^*$ to an invariant subspace. Then $\zeta$ is an eigenvalue of $\fb _{T,x} (A_\zeta ^\mrt )$.
\end{prop*}

There is a minor error in the proof of \cite[Proposition 5.6]{JMS-ncratClark}. Fortunately, the proposition statement is still correct, and the proof can be readily fixed: 

\begin{proof}{ (\cite[Proposition 5.6]{JMS-ncratClark})}
We have that $\zeta \in \partial \D \cap \sigma (\fb (Z) )$ if and only if 
$$ \mr{det} \left( I \otimes I - \sum_{j=1}^d Z_j \otimes T(\zeta )_j ^* \right) =0. $$ Taking adjoints, this happens if and only if 
$$ \mr{det} \left( I \otimes I - \sum_{j=1}^d Z_j^* \otimes T(\zeta)_j \right) =0. $$ By vectorization, this happens if and only if there is a matrix, $X$, so that 
$$ X - \sum_{j=1}^d T(\zeta )_j X \ov{Z}_j =0. $$ If 
$$ T(\zeta ) ^* =: \bpm A^* & B^* \\  & C^* \epm, $$ then take 
$$ Z  = \bpm A^\mrt  &  0 \\ 0 & 0 \epm, $$ and if $A \in \C ^d _k$, let $X = I_k$ so that 
\ba  X - T(\zeta ) X \ov{Z} & = & \bpm I_k & 0 \\ 0 & 0 \epm - \bpm A & 0 \\ B & C \epm \bpm I_k & 0 \\ 0 & 0 \epm \bpm A^* & 0 \\ 0 & 0 \epm \\
& = & \bpm I_k & 0 \\ 0 & 0 \epm - \bpm AA ^* & 0 \\ BA^* & 0 \epm = \bpm I_k & 0 \\ 0 & 0 \epm - \bpm I_k & 0 \\ 0 & 0 \epm =0. \ea 
Here we have used the fact that $T(\zeta)$ is a row coisometry.
\end{proof}

\subsection{Peaking states and representations} \label{subsec:peak}

In this section we follow the work of Clou\^{a}tre and Thompson in non-commuative convexity in operator algebra and operator system theory \cite{ClouThom-min_bound}. Let $A$ be a unital operator algebra and let $B$ be a $C^*$-cover of $A$. Namely, we have a unital, completely isometric embedding $\iota \colon A \to B$, such that $B = C^*(\iota(A))$. Let us denote by $K(B)$ the state space of $B$. Clou\^{a}tre and Thompson say that $\mu \in K(B)$ is \emph{$A-$peaking}, or an \emph{$A-$peak state}, if there exists a contraction $a \in A$, such that $\mu(a^* a) = 1$ and $\nu(a^* a) < 1$, for all $\nu \in K(B) \setminus \{\mu\}$. If $\mu$ is $A$-peaking, then it is quite immediate that $\mu$ is pure and $\mu$ has the unique extension property (UEP), i.e., the functional $\mu|_A$ has a unique Hahn-Banach extension to $B$. In fact, any $\mu \in \posncm$ with the property that its GNS row isometry is a Cuntz row isometry has the UEP by \cite[Proposition 5.11]{JMT-ncFnM}. If $\pi_{\mu} \colon B \to B(\cH_{\mu})$ is a GNS representation of $\mu$ with cyclic vector $\xi_{\mu}$, then, by Cauchy-Schwarz, $\pi_{\mu}(a) \xi_{\mu} = \xi_{\mu}$. This implies that there is a finite-dimensional subspace $F \subset \cH_{\mu}$, such that $\|P_F \pi_{\mu}(a)|_F\| = 1$, where $P_F$ is the projection onto $F$. Similarly, in \cite{Clouatre-peaking}, if $B$ is a $C^*-$algebra and $\scr{S} \subseteq B$ is an operator system, a state, $\mu \in K (B)$ is said to be $\scr{S}-$peaking if there is a self-adjoint element $s \in \scr{S}$ so that $\| s \| = 1$ and $|\la (s)| < \mu (s) =1$ for every $\la \in K (B) \sm \{ \mu \}$.

An irreducible representation $\pi \colon B \to B(\cH_{\pi})$ is called local $A$-peak by Clou\^{a}tre and Thompson, if there exists $n \in \N$, $T \in M_n(A) := A \otimes M_n$ with $\|T\|=1$, such that for every irreducible representation $\sigma \colon B \to B(\cH_{\sigma})$ unitarily inequivalent to $\pi$ and every finite-dimensional subspace $G \subset \cH_{\sigma}^{\oplus n}$, $\|P_G \sigma^{(n)}(T)|_G\| < 1$. This implies, that there is a finite-dimensional subspace $F \subset \cH_{\pi}^{\oplus n}$, such that $\|P_F \pi^{(n)}(T)|_F \| = 1$ \cite[Lemma 2.2]{ClouThom-min_bound}. By \cite[Theorem 2.7]{ClouThom-min_bound}, if $\mu \in K(B)$ is $A$-peaking, then $\pi_{\mu}$ is local $A$-peak. The converse, however, is false \cite[Example 2]{ClouThom-min_bound}. One can say slightly more about the connection between the two notions. This is the goal of the following two lemmas.

\begin{lem}
Let $B$ be a unital $C^*$-algebra and $A \subset B = C^*(A)$ be a unital operator algebra. Let $\pi \colon B \to B(\cH)$ be a representation, such that there exists a contraction $a \in A$ and a finite-dimensional subspace $G \subset \cH$ with $\|P_G \pi(a)|_G\| = 1$. Then, there exists a vector $\xi \in G$, such that $\pi(a^* a) \xi = \xi$.
\end{lem}
\begin{proof}
Let $\varphi$ be the ucp $\varphi(b) = P_G \pi(b)|_G \in B(G)$. By Kadison--Schwarz, $\varphi(a)^* \varphi(a) \leq \varphi(a^* a)$. Hence,
\[
1 = \|\varphi(a)\|^2 \leq \|\varphi(a^* a)\| \leq 1
\]
In particular, $\|\varphi(a^*a)\| = 1$ and since it is a selfadjoint operator on a finite-dimensional space, there exists $\xi \in G$, such that $\varphi(a^* a) \xi = \xi$. Since $\pi(a^* a)$ is a contraction, we have that $\pi(a^*a) \xi = \xi$.
\end{proof}

\begin{lem}
Let $B$ be a unital $C^*$-algebra and $1 \in A \subset B$ an operator algebra, such that $B = C^*(A)$. Let $a \in A$ be a contraction and set
\[
F_a = \left\{\mu \in K(B) \mid \mu(a^* a) = 1 \right\}.
\]
Then, $F_a$ is a closed face of $K(B)$. Moreover, if $\pi \colon A \to B(\cH)$ is local $A$-peak with witness $a$, then $F_a \neq \emptyset$ and all the extreme points of $F_a$ are states arising from $\pi$. Lastly, $\partial_e F_a $ is in one to one correspondence with vectors $\eta \in \cH$, such that $\pi(a^* a) \eta = \eta$.
\end{lem}
\begin{proof}
Since $a^* a$ defines a weak*-continuous contractive real functional on $(B_{sa})^*$, then $F_a$ is the closed face where this functional attains its maximum on $K(B)$. Let $\pi$ be a local $A$-peak representation with witness $a$. Then, by the previous lemma, there exists $\xi \in G$, such that $\pi(a^* a) \xi = \xi$. Hence, the state $\mu(b) = \langle  \xi, \pi(b) \xi \rangle$ is in $F_a$. Since $F_a$ is a closed face, its extreme points are pure states. Let $\varphi \in F_a$ be a pure state and $(\sigma, \cK, \eta)$ be its GNS representation. Then, it is immediate that $\sigma(a^* a) \eta = \eta$. Thus, if $G' \subset \cK$ is the space spanned by $\eta$ and $\sigma(a) \eta$, then $\|P_{G'} \sigma(a)|_{G'}\| = 1$ and thus $\sigma$ is unitarily equivalent to $\pi$
\end{proof}

In this paper, we are interested in the case when $A = \A$ and $B = \cuntz$. Since $ \A $ is semi-Dirichlet, we can say more about $ \A $ -peak states and connect them to peaking states for operator systems.

\begin{lem} \label{lem:exposed}
If $\mu \in K(\cuntz)$ is an $\A$-peak state, then there exists $b \in \scr{A}_d$ positive, such that $\mu(b) = 1$ and $\nu(b) < 1$, for all states $\nu \neq \mu$. In particular, $\mu$ is an $\scr{A} _d-$peak state, $\mu|_{\scr{A}_d}$ is a weak* exposed extreme point of the unit ball of $(\cS_d^*)_{sa}$ and it has the unique extension property. 
\end{lem}
\begin{proof}
Let $\mu \in K(\cuntz)$ be an $\A$-peak state. Then, there is a contraction $a \in \scr{A} _d$, such that $\mu(a^* a) = 1$ and for every state $\nu \neq \mu$, $\nu(a^* a) < 1$. However, since $A$ is semi-Dirichlet, $a^* a \in \scr{A}_d$. Hence, we are done. The second part follows from \cite[Theorem 3.2]{Clouatre-peaking}.
\end{proof}

One could try and argue the converse claim. Let $b \in \scr{A}_d$ be a positive element, such that $\mu(b) = 1$ and for every $\nu \in K(\cuntz) \setminus \{\mu\}$, $\nu(b) < 1$. By replacing $b$ by $\frac{1}{2}(1 + b)$, we may assume that $b$ is invertible. Thus, $b$ is factorizable in the sense of Popescu \cite{Pop-factor}. Namely, there exists $c \in \bH^{\infty}_d$, such that $b = c^* c$. However, we do not know that $c \in \A$. This is true if we assume that $b$ is the real part of an NC rational function. In this case, $c$ is NC rational by the NC rational Fej\'er--Riesz theorem of \cite[Theorem 6.5]{JM-subFock}. However, we do not need this result for our purposes and only require the following observation to construct a large class of examples of $\A$-peaking states.
\begin{lem}
Let $b \in \A$ be inner. Let $\mu \in K(\cuntz)$ be such that $\mu(b) = 1$ and for all $\nu \in K(\cuntz) \setminus \{\mu\}$, $|\nu(b)| < 1$. Then, $\mu$ is $\A-$peaking.
\end{lem}
\begin{proof}
Since $b$ is inner, we have that:
\[
\frac{1}{4} (1 + b^*) (1 + b) = \frac{1}{4}(2 + 2 \Re(b)) = \frac{1}{2}(1 + \Re(b)).
\]
Let $a = \frac{1}{2}(1 + \Re(b))$, then $a \geq 0$ and $\mu(a) = 1$. Moreover, if $\nu \in K(\cuntz) \setminus \{\mu\}$ is such that $\nu(a) = 1$, then $\Re \nu(b) = 1$. However, this contradicts the assumption that $|\nu(b)| < 1$.
\end{proof}

The above observation suggests the following strategy: To construct examples of states on the Cuntz algebra which peak on the free disk algebra, we will show that if $\fb \in \mult$ is NC rational and inner, then there are finite points, $A$, on the boundary of the unit row-ball so that $\fb (A)$ has $1$ as an eigenvalue of multiplicity one.

\section{Main result} \label{sec:main}

A relationship between spectra and intertwiners is encoded in the following lemmas. Let $Z \in \C^d_n$ be an irreducible row contraction of row norm $1$. Let $Y \in B (\cH)^d =\C ^d _\infty$ be another row contraction on a separable Hilbert space. Let $\psi_{Y,Z} \colon B(\C^n, \cH) \to B(\C^n, \cH)$ be the map $\psi_{Y,Z}(T) = \sum_{j=1}^d Y_j T Z_j^*$. Assume that $1 \in \sigma_p(\psi_{Y,Z})$. Then, there exists $0 \neq T \in B(\C^n, \cH)$, such that $\psi_{Y,Z}(T) = T$. We may assume that $\|T\| = 1$. Then, since $n < \infty$, there is a unit vector $v \in \C^n$, such that $\|T v\| = 1$ and hence
\[
1 = \|T v\|^2 = \langle \psi_{Y,Z}(T) v, T v \rangle = \langle (I \otimes T) Z^* v, Y^* T v \rangle.
\]
Since $T$, $Z$, and $Y$ are contractions, we conclude from the equality clause of Cauchy-Schwarz that $Y^* T v = (I \otimes T) Z^* v$. Therefore, for every $1 \leq n \leq d$. The same argument can be applied to the self-compositions $\psi_{Y,Z}^{\circ n}$ to conclude that for every $p \in \fp$, $T p(Z^*) v = p(Y^*) T v$. Now since $Z$ is irreducible, so is $Z^*$. This implies that $v$ is cyclic. Therefore, for every $w \in \C^n$, there exists an NC polynomial $p$, such that $p(Z^*) v = w$. Hence, for every $1 \leq j \leq d$,
\[
T Z_j^* w = T Z_j^* p(Z^*) v = Y_j^* p(Y^*) T v = Y_j^* T p(Z^*) v = Y_j^* T w.
\]
We conclude that $T$ is a homomorphism of $\fp$-modules from $\C^n$ to $\cH$. In particular, since the kernel of a homomorphism is a submodule. $T$ must be injective. Thus, we have obtained the following lemma:

\begin{lem} \label{lem_spectrum_to_embedding}
Let $Z \in \C^d_n$ be an irreducible row contraction of row norm $1$ and $Y \in B(\cH)^d$ a row contraction, such that $1 \in \sigma_p(\psi_{Y,Z})$. Then, there exists a $Y$-coinvariant $n$-dimensional subspace $\cK \subset \cH$, such that $Y^*|_{\cK}$ is similar to $Z^*$.
\end{lem}

We can do slightly better, assuming that both $Z$ and $Y$ are coisometries. Under the assumptions of the lemma, let $T$ be the intertwinner obtained above. Then,
\[
Z (I \otimes T^* T) Z^* = T^* Y Y^* T = T^* T.
\]
However, since $Z$ is irreducible, a result of Farenick \cite[Theorem 2]{Farenick-irr_pos} implies that the map $A \mapsto Z(I \otimes A) Z^*$ is irreducible and thus, by the Perron-Frobenius theorem for positive maps of Evans and H\o{}egh-Krohn \cite[Theorem 2.3]{EHK-PerronFrobenius}, we know that there is a unique (up to scalar multiplication) eigenvector of this map that corresponds to eigenvalue $1$. However, since $Z$ is a coisometry, the corresponding map is a ucp. Hence, $T^* T$ is a scalar multiple of the identity of norm $1$. Therefore, $T^* T = I$ and $T$ is an isometry. We summarize

\begin{lem} \label{lem:spectrum_to_isometry}
Let $Z \in \C^d_n$ be an irreducible row coisometry and let $Y \in B(\cH)^d$ be a row coisometry. If $1 \in \sigma_p(\psi_{Y,Z})$, then there exists a unique isometry $V \colon \C^n \to \cH$, such that $Y^* V = (I \otimes V) Z^*$.
\end{lem}

Note that we can canonically identify $B(\C^n, \cH)$ with $(\C^n)^* \otimes \cH$. The identification sends $\varphi \otimes \xi$ to the linear map $v \mapsto \varphi(v) \xi$. Now let $Y \in B(\cH)^d$ and $X \in \C^d_n$. Then, we can define an operator $\sum_{j=1}^d X_j^t \otimes Y_j \in B((\C^n)^* \otimes \cH)$ acting via
\[
\left(\sum_{j=1}^d X_j^t \otimes Y_j \right) (\varphi \otimes \xi) = \sum_{j=1}^d (\varphi \circ X_j) \otimes Y_j \xi.
\]
Now via identification, the right-hand map is
\[
\left(\sum_{j=1}^d (\varphi \circ X_j) \otimes Y_j \xi\right) v = \sum_{j=1}^d \varphi(X_j v) Y_j \xi = 
\left(\sum_{j=1}^d Y_j (\varphi \otimes \xi) X_j\right) v.
\]
Therefore, we get that the map $\psi_{Y,X}$ corresponds to the product $\sum_{j=1}^d \overline{X_j} \otimes Y_j$. Hence, $1 \in \sigma_p(\psi_{Y,X})$ if and only if $1 \in \sigma_p \left( \sum_{j=1}^d \overline{X_j} \otimes Y_j \right)$,

Assume that $T \in \C ^d _n$ is an irreducible row co-isometry. By \cite[Lemma 2.1]{JMS-ncratClark}, $\mr{row} (T ^\mrt )$ has joint spectral radius $1$ and by \cite[Lemma 4.10]{SSS-bounded}, $\mr{row} (T ^\mrt )$ is jointly similar to an irreducible row co-isometry, $Z \in \B ^d _n$. Fix a unit vector, $x \in \C ^n$. Then there is a unique NC rational inner function, $\fb = \fb _{T, x}$, corresponding to the pair $(T,x)$ by \cite[Theorem 4.1]{JMS-ncratClark}. Let $S \in \GL_n$ be an invertible matrix, such that $S^{-1} T^t S = Z$. Since $\fb$ is NC, we have that $$\fb(Z) = \fb(S^{-1} T^t S) = S^{-1} \fb(T^t) S.$$
In particular, the dimension of the eigenspace corresponding to $1$ of $\fb(T^t)$ and $\fb(Z)$ are the same. 

\begin{lem} \label{lem:Z_eigenspace}
In the above setting, $\fb(Z)$ has an eigenspace of dimension $1$ corresponding to eigenvalue $1$.
\end{lem}
\begin{proof}
Consider the map $\psi_T(X) = \sum_{j=1}^d T_j X T_j^*$. This map corresponds to the tensor $\sum_{j=1}^d \overline{T_j} \otimes T_j$. Since the map is unital, we have that $1 \in \sigma\left( \sum_{j=1}^d \overline{T_j} \otimes T_j \right)$. Moreover, $1$ is the Perron-Frobenius eigenvalue of $\psi_T$, and since $T$ is irreducible, the dimension of the corresponding eigenspace is $1$. Taking the adjoint, we get the tensor $\sum_{j=1}^d T_j^t \otimes T_j^*$. By Lemma \ref{lem:spectrum_of_schur_complement}, we know that the dimension of the eigenspace corresponding to $1$ of $\fb(T^t)$ is $1$, as well. The claim for $Z$ now follows from the observation preceding the lemma.
\end{proof}
Let $y$ be a unit vector in this one-dimensional eigenspace of $\fb (Z)$ to eigenvalue $1$ and define the finitely--correlated Cuntz state $\mu = \mu _{Z,y} \in \posncm$ by $\mu (L^\om ) := y^* Z^\om y$. Let $V$ denote the minimal row isometric dilation of $Z$ on $\cH \supseteq \C ^n$. By \cite[Theorem 6.5]{DKS-finrow}, $V$ is an irreducible Cuntz row isometry. Given any $\mu \in \posncm$, we can, as in \cite{JMS-ncratClark}, apply a Gelfand--Naimark--Segal (GNS) construction to $(\mu , \A)$ to obtain a GNS-Hilbert space, $\hardy (\mu)$, and a row isometry, $\Pi _\mu$, acting on $\hardy (\mu )$, where $\hardy (\mu )$ is defined as the completion of the free algebra, $\fp$, modulo vectors of zero--length, with respect to the pre-inner product, 
$$ \ip{p}{q}_\mu := \mu (p (L) ^* q(L) ). $$ Equivalence classes of free polynomials, $p + N_\mu$, where $N_\mu$ denotes the left ideal of zero-length vectors with respect to $\| \cdot \| _\mu$, are dense in $\hardy (\mu )$ by construction. This construction also comes equipped with a row isometry, $\Pi _\mu$, defined by left multiplications by the independent variables on $\hardy (\mu)$, $\Pi _{\mu ;j} p + N_\mu := \mf{z} _j p + N_\mu$. Also note that defining $\pi _\mu (L_k ) := \Pi _{\mu ; k }$, $1 \leq k \leq d$, yields a $*-$representation of the Cuntz--Toeplitz algebra, $\scr{E} _d$, and vice versa. Moreover, $\Pi _\mu$ is a Cuntz (surjective) row isometry if and only if $\mu \in \posncm$ has a unique positive extension to $\cuntz$ by \cite[Proposition 5.11]{JMT-ncFnM}. Hence, any $\mu \in \posncm$ with GNS row isometry $\Pi _\mu$, of Cuntz type can be uniquely identified with a positive state, $\hat{\mu}$, on $\cuntz$, and applying the GNS construction to $(\mu, \A)$ or to $(\hat{\mu}, \cuntz )$ yields the same GNS-Hilbert space, $\hardy (\mu )$, and the same Cuntz representation, $\pi _\mu$.

By \cite[Proposition 3.6]{JMS-ncratClark}, if $\mu = \mu _{Z,y}$ and we define
$$ \cH _\mu := \bigvee \Pi _\mu ^{*\om} 1 + N _\mu, \quad \quad Z_\mu := (\Pi _\mu ^* | _{\cH _\mu } ) ^*, $$ then the pair $(Z, V )$ are jointly unitarily equivalent to $(Z_\mu ,\Pi _\mu )$ by a unitary which sends $y$ to $1 + N_\mu$. In particular, $\cH _\mu$ is finite--dimensional, $Z_\mu$ is an irreducible row co-isometry, and $\Pi _\mu$ is its minimal row isometric dilation, and this is irreducible and Cuntz. 

\begin{lem} \label{intertwine2}
Let $\pi \colon \cuntz \to B(\cH)$ be a representation so that $1 \in \sigma _p (\pi (\fb ) )$. Then, there exists a unique isometry $V \colon \C^n \to \cH$, such that $\pi(S)^* V = (I \otimes V) Z^*$.
\end{lem}
\begin{proof}
Since $\xi = \pi(\fb) \xi = \fb(\pi(S)) \xi$, by Proposition \ref{lem:spectrum_of_schur_complement}, we have that $1 \in \sigma_p\left(\sum_{j=1}^d \pi(S_j) \otimes T_j^*\right)$. Now applying the interchange unitary, we have that $1 \in \sigma_p\left(\sum_{j=1}^d T_j^* \otimes \pi(S_j)\right)$. The latter corresponds to the map $\psi \colon B(\C^n, \cH) \to B(\C^n, \cH)$ given by $\psi(X) = \sum_{j=1}^d \pi(S_j) X \overline{T_j}$. We note that $\bar{T} = (T^t)^*$. However, we also know that $\overline{T_j} = S^* Z_j^* S^{-1 *}$. Thus,
\[
\psi(X) = \sum_{j=1}^d \pi(S_j) X \overline{T_j} = (I \otimes S^*) \left(\sum_{j=1}^d \pi(S_j) X Z_j^* \right) (I \otimes S^{-1 *}).
\]
Hence, if we set $\varphi(X) = \sum_{j=1}^d \pi(S_j) X Z_j^*$. Then, $1 \in \sigma_p(\varphi)$. Therefore, by Lemma \ref{lem:spectrum_to_isometry}, we get that there exists a unique isometry $V \colon \C^n \to \cH$, such that $\pi(S)^* V = (I \otimes V) Z^*$.
\end{proof}

\begin{thm} \label{thm:main}
Let $T \in \C ^d _n$ be an irreducible row co-isometry and $x \in \C ^n$ a unit vector so that $\fb = \fb _{T,x}$ is the unique NC rational inner function corresponding to $(T,x)$. Let $Z$ be the irreducible row co-isometry which is jointly similar to $\mr{row} (T ^\mrt )$ via an invertible matrix, $S$, and let $y \in \C ^n$ be the unique unit eigenvector of $\fb (Z)$ to eigenvalue $1$.  Then the finitely--correlated Cuntz state $\mu := \mu _{Z,y} \in \posncm$ is an $\A-$peak state which peaks at $\fb \in \A$.
\end{thm}

\begin{proof}
As described above, since $\mu (L^\om ) := y^* Z^\om y$ and $\fb (Z) y = y$, where the unit vector, $y$, spans the eigenspace for $\fb (Z)$ corresponding to eigenvalue $1$, it follows that $\mu (\fb ) = y^* \fb (Z) y = 1$. Note that $\mu$ is a pure Cuntz state since $\pi _\mu$ is an irreducible representation of the Cuntz algebra. By the equality in the Cauchy--Schwarz inequality, it follows as before that $\pi _\mu (\fb ) 1 + N_\mu = \fb (\Pi _\mu) 1 + N_\mu = 1 + N_\mu$. Assume that $h \in \hardy (\mu )$ is any other element such that $\fb (\Pi _\mu ) h = h$.  Note that $\cH _\mu$ is $\Pi _\mu-$co-invariant by construction so that $\cH _\mu ^\perp$ is invariant. Consider the block decomposition of $\Pi _\mu$ and $h$ with respect to $\hardy (\mu ) = \cH _\mu \oplus \cH _\mu ^\perp$:
$$ \fb (\Pi _\mu ) h = \bpm \fb (Z_\mu ) & 0 \\ * &  * \epm \bpm h_1 \\ h_2 \epm = \bpm h_1 \\ h_2 \epm, $$ and we conclude that $\fb (Z_\mu )h_1 = h_1$. Since $y$ is the unique (up to scalars) eigenvector of $\fb (Z)$ to eigenvalue $1$, $1 + N_\mu$ is the unique eigenvector of $\fb (Z_\mu )$ so that $h_1 = \alpha 1 + N_\mu$ for some $\alpha \in \C$. It further follows that $h_2 = h - \alpha (1 + N_\mu ) \in \cH _\mu ^\perp$ is an eigevector of $\fb (\Pi _\mu ) | _{\cH _\mu ^\perp}$ to eigenvalue $1$. By \cite[Corollary 5.3]{DKS-finrow}, $\cH _\mu ^\perp \simeq \hardy \otimes \C ^k$ and $\Pi _\mu | _{\cH _\mu ^\perp} \simeq L \otimes I_k$. However, since $P_{\cH_{\mu}^{\perp}} \fb (\Pi _\mu ) | _{\cH _\mu ^\perp} \simeq \fb (L) \otimes I_k$, this is a pure isometry and we conclude that $h_2 =0$. That is, $1 + N_\mu$ is the unique eigenvector (up to non-zero scalars) of $\fb (\Pi _\mu )$ to eigenvalue $1$.

Now let $\la \in K(\cuntz)$ be another Cuntz state which peaks at $\fb$, $\la (\fb ) =1$. As before equality in the Cauchy--Schwarz inequality implies that $\fb (\Pi _\la ) 1 + N_\la = 1 + N_\la$. By Lemma \ref{intertwine2} and Lemma \ref{lem:spectrum_to_isometry}, there is a unique isometry, $V : \C ^n \rightarrow \hardy (\la )$ so that 
$$ V Z^{\om *} = \Pi _\la ^{\om *} V. $$

Hence, 
$$ V^* \Pi _\la ^\om V = Z^\om, $$ and $\Pi _\la$ is a row-isometric dilation of $Z$, which must be minimal as $\Pi _\la$ is irreducible. Hence, by uniqueness of the minimal dilation, $\Pi _\la \simeq \Pi _\mu$. Let $U : \hardy (\la ) \rightarrow \hardy (\mu)$ be the unitary implementing this equivalence, $U \Pi _\la ^\alpha = \Pi _\mu ^\alpha U$. Then,
\ba U 1 + N_\la & = & U \fb (\Pi _\la ) 1 + N_\la \\
& = & \fb (\Pi _\mu ) U 1 + N_\la, \ea so that $U 1 + N_\la =:h$ is a unit eigenvector of $\fb (\Pi _\mu )$ to eigenvalue $1$ so that by the previous arguments, $U 1 + N_\la = \zeta (1 + N_\mu )$ for some $\zeta \in \partial \D$. This proves that $\la = \mu$ so that $\mu = \mu _{Z,y}$ is an $\A-$peak state. 
\end{proof}

The following corollary follows immediately from the preceding theorem and Lemma \ref{lem:exposed}.

\begin{cor}
Every finitely--correlated state $\mu$ on $\scr{A}_d$ that arises from an irreducible finite-dimensional row coisometry and a unit vector is an exposed extreme point of the state space of $\scr{A}_d$.

\end{cor}

\begin{cor}
If $\mu = \mu _{T,x}$, for a finite irreducible row coisometry $T$ and unit vector $x$, let $Z$ be the unique irreducible finite row co-isometry which is jointly similar to $\mr{row} (T ^\mrt )$ and let $y$ be the eigenvector of $\fb _{T,x} (Z)$ corresponding to the multiplicity one eigenvalue, $1$. Then the finitely--correlated state $\mu _{Z,y}$ peaks at the NC rational inner $\fb _{T,x}$ and $\mu _{T,x}$ peaks at the NC rational inner $\fb _{Z,y}$.
\end{cor}

If we define the unital, completely positive map,
$$ \mr{Ad} _{T,T^*} (A) := \sum _{j=1} ^d T_j A T_j ^*, $$ then this map is a unital quantum channel, \emph{i.e.} $\mr{Ad} _{T^* , T}$ is also unital, if and only if $\mr{row} (T ^*)$ is also a row coisometry, or, again equivalently, $\mr{row} (T^\mrt)$ is a row coisometry. We require another definition to describe the class of NC rational inner functions associated with unital quantum channels. Let $\alpha = i_1 \cdots i_n$ be a word in the alphabet $\{1,\cdots,d\}$. We set $\alpha^t = i_n i_{n-1} \cdots i_1$. Namely, $\alpha^t$ is the reversal of $\alpha$. We define a unitary on $\bH^2_d$ by $(\mathfrak{z}^{\alpha})^t = \mathfrak{z}^{\alpha^t}$. This unitary is important in realization theory of NC rational functions \cite{JMS-NCrat}. It is proved in \cite[Lemma 2.2]{JMS-ncratClark} that if $\fr$ is an NC rational function, so is $\frt$. However, it need not be the case that if $\fb$ is an NC rational inner, that $\fb^t$ is inner. A simple example is $\fb(\mathfrak{z}) = (1 + \mathfrak{z_1})\mathfrak{z_2}$. It is an immediate calculation, that $\fb(L)^* \fb(L) = I$. However, $\fb^t(\mathfrak{z}) = \mathfrak{z_2}(1 + \mathfrak{z_1})$. Here the inner part of $\fb^t(L)$ is $L_2$, and the outer part is $1 + L_1$. In particular, it is easily checked that $\| \fbt (L) \| = \sqrt{2}$ so that $\fbt \in \mult$ is not even contractive, see \cite[Example 3.4]{JM-ncld}.  The following theorem identifies the class of all NC rational functions $\fb$, such that $\fb^t$ is also inner, as precisely those that arise from quantum channels.

\begin{thm} \label{thm:qc}
Let $\fb := \fb _{T,x}$ be the NC rational inner generated by the pair $(T,x)$, where $T$ is a finite--dimensional row co-isometry on $\cH$ and $x \in \cH$ is both $T$ and $T^*-$cyclic. Then, $\mr{row} (T^\mrt)$ is also a row co-isometry if and only if $\fbt$ is also NC rational inner and in this case $\fbt = \fb _{\mr{row} (T^\mrt) , \ov{x}}$.  If $x$ is a unit vector and both $T$ and $T^\mrt$ are irreducible row coisometries then the NC rational Clark states $\mu _{T,x}$ and $\mu _{T ^\mrt , \bar{x}}$ peak at the NC rational inners $\fb ^\mrt$ and $\fb$, respectively.
\end{thm}

\begin{proof}
By \cite[Theorem 3.2, Remark 3.4]{JMS-ncratClark}, $(T^*,x,x)$ is a minimal descriptor realization of $\mf{G} := (1 - \fb ) ^{-1}$, so that the Taylor coefficients of $\mf{G}$ at $0$ are $\hat{\mf{G}} _\om = x^* T^{*\om} x$. Hence the Taylor coefficients of $\mf{G} ^\mrt = (1 - \fbt ) ^{-1}$ are equal to
\ba \mf{G} ^\mrt _\om & = & \hat{\mf{G}} _{\om ^\mrt} = x^* T ^{* \om ^\mrt} x  \\
& = & x^* (T^{\om}) ^* x  \\
& = & (T^\om x) ^* x \\ 
& = & \ov{x} ^* \ov{T} ^\om \ov{x} \\
& = & \ov{x} ^*  T^{\mrt * \om} \ov{x}. \ea 
This shows that $\mf{G} ^\mrt$ has the minimal descriptor realization $( \ov{T} , \ov{x} , \ov{x} )$, where $T^{\mrt *} = \ov{T}$. If $\mr{row} (T ^\mrt)$ is also finite row coisometry then by \cite[Theorem 4.1]{JMS-ncratClark}, it follows that $\fbt$ is also NC rational inner with minimal FM realization:
$$ A _j := \ov{T} _j (I-\ov{x}\ov{x}^*), \quad B_j := \ov{T} _j \ov{x}, \quad C := \ov{x}^*, \quad \mbox{and} \quad D:= \fbt (0) =0. $$

Conversely, as above, given any NC rational inner $\fb = \fb _{T,x}$, a minimal descriptor realization of $(1 - \fbt ) ^{-1} = \mf{G} ^\mrt$ is given by $(\ov{T}, \ov{x} , \ov{x} )$.  Assuming that $\fbt$ is also NC rational inner, \cite[Theorem 3.2, Remark 3.4]{JMS-ncratClark} implies that there is a finite row co-isometry, $W$, and a vector, $y$, which is both $W$ and $W^*-$cyclic so that $(W,y, y)$ is also minimal descriptor realization of $\mf{G} ^\mrt$ so that for any word, $\om \in \F ^d$, 
$$ y^* W ^{*\om} y = \ov{x} ^* \ov{T} ^\om \ov{x}. $$ 
Equivalently, if $\| T ^\mrt \| _{row}$ is the row-norm of $\mr{row} (T^\mrt )$, then for any word, $\om \in \F ^d$, 
$$ \frac{1}{\| T^\mrt  \| _{row} ^{| \om | } } y^* W ^{*\om} y =  \frac{1}{\| T^\mrt  \| _{row} ^{| \om | } } \ov{x} ^* \ov{T} ^\om \ov{x}. $$
If $\| \mr{row} (T ^\mrt ) \| > 1$, then $\frac{1}{\| T ^\mrt \|_{row} } \mr{row} (T ^\mrt )$ and $\frac{1}{\| T ^\mrt \| _{row}} W$ are both row contractions. In either case, \cite[Proposition 3.6, Lemma 3.9]{JMS-ncratClark} implies that $\mr{row} (T^\mrt )$ and $W$ are jointly unitarily equivalent via a unitary $U$ which sends $y$ to $\ov{x}$. Hence $\mr{row} (T^\mrt )$ is a row coisometry. 
\end{proof}

\begin{example}
The examples \cite[Example 4.4, Example 4.5]{JMS-ncratClark} both give examples of NC rational inners arising from finite, irreducible row coisometries, $T$. Namely,
$$ T = \left( \bpm 0 & 1 \\ 0 & 0 \epm  , \bpm 0 & 0 \\ 1 & 0 \epm \right)$$
and 
$$ S = \frac{1}{\sqrt{2}} \left( \bpm 1 & 0 \\ 0 & -1 \epm , \bpm 0 & -1 \\ 1 & 0 \epm  \right). $$
It is easily checked that $\mr{row} (T^\mrt)$ and $\mr{row} (S^\mrt )$ are both row coisometries so that, by the previous theorems, if $x$ is any unit vector then $\fb _{T,x}, \fbt _{T,x} = \fb _{\mr{row} (T^\mrt ) , \ov{x} }$ and $\fb _{S,x}, \fbt _{S,x}$ are all NC rational inner, $\mu _{T,x}, \mu _{\mr{row} (T ^\mrt ) , \ov{x}}$ and $\mu _{S,x}, \mu _{\mr{row} (S ^\mrt ) , \ov{x}}$ are all Cuntz states which peak at $\fbt _{T,x}, \fb _{T,x}, \fbt _{S,x}$ and $\fb _{S,x}$, respectively.
\end{example}

The following example illustrates what happens if we drop the assumption that $T$ is irreducible but require still that its minimal isometric dilation is irreducible.

\begin{example}
Consider the following coisometry:
\[
T_1 = \dfrac{1}{2}\begin{pmatrix} -1 & 0 & -1 \\ -1 & 0 & 1 \\ -1 & 0 & -1 \end{pmatrix},\quad T_2 = \dfrac{1}{2} \begin{pmatrix} 1 & -1 & 0 \\ -1 & -1 & 0 \\ -1 & 1 & 0 \end{pmatrix}.
\]
Let us write $e_1$, $e_2$, and $e_3$ for the vectors of the standard basis of $\C^3$. I is now easy to check that $T_1 e_1 = -\frac{1}{2}(e_1 + e_2 + e_3)$, $T_2 e_1 = \frac{1}{2} (e_1 - e_2 - e_3)$, $T_2 T_1 e_1 = \frac{1}{2} e_2$. This implies that $e_1$ is $T$-cyclic. Similarly, $T_1^* e_1 = -\frac{1}{2}(e_1 + e_3)$, $T_2^* e_1 = \frac{1}{2}(e_1 - e_2)$, and $T_1^* T_2^* e_1 = - \frac{1}{2} e_3$. This implies that $e_1$ is also $T^*$-cyclic and, moreover, that $\bigvee_{\alpha \neq \emptyset} T^{* \alpha} e_1 = \C^3$. Therefore, by \cite{JMS-ncratClark}, the following NC rational function is an inner
\[
\fr(\mathfrak{z}_1,\mathfrak{z}_2) = e_1^* \left(I - \mathfrak{z}_1 T_{1,0}^* - \mathfrak{z}_2 T_{2,0}^*\right)^{-1} \left(\mathfrak{z}_1 T_1^* e_1 + \mathfrak{z}_2 T_2^* e_1\right).
\]
Here,
\[
T_{1,0}^* = T_1^*(I - e_1 e_1^*) \text{ and } T_{2,0}^* = T_2^*(I - e_1 e_1^*).
\]
Therefore, we have the following expression for the pencil
\[
I - \mathfrak{z}_1 T_{1,0}^* - \mathfrak{z}_2 T_{2,0}^* = \begin{pmatrix} 1 & \frac{1}{2}(\mathfrak{z}_1 + \mathfrak{z}_2) & \frac{1}{2}(\mathfrak{z}_1 + \mathfrak{z}_2) \\ 0 & 1 + \frac{1}{2} \mathfrak{z}_2 & -\frac{1}{2} \mathfrak{z}_2 \\ 0 & - \frac{1}{2} \mathfrak{z}_1 & 1 + \frac{1}{2} \mathfrak{z}_1 \end{pmatrix}.
\]
Since we know $T_1^* e_1$ and $T_2^* e_1$, we conclude that the expression for our function is
\[
\fr(\mathfrak{z}_1,\mathfrak{z}_2) = \frac{1}{2} (\mathfrak{z}_2 - \mathfrak{z}_1) + \frac{1}{4} (\mathfrak{z}_1 + \mathfrak{z}_2) \begin{pmatrix} 1 & 1 \end{pmatrix} \begin{pmatrix} 1 + \frac{1}{2} \mathfrak{z}_2 & -\frac{1}{2} \mathfrak{z}_2 \\ - \frac{1}{2} \mathfrak{z}_1 & 1 + \frac{1}{2} \mathfrak{z}_1 \end{pmatrix}^{-1} \begin{pmatrix} \mathfrak{z}_2 \\ \mathfrak{z}_1 \end{pmatrix}.
\]
In particular, we have that $r(-1,0) = 1$.

Now we observe that $T_1^*$ and $T_2^*$ have a common eigenvector. In fact, set
\[
U = \begin{pmatrix} \frac{1}{\sqrt{2}} & 0 & \frac{1}{\sqrt{2}} \\ 0 & 1 & 0 \\ \frac{1}{\sqrt{2}} & 0 & -\frac{1}{\sqrt{2}} \end{pmatrix}.
\]
Then,
\[
U T_1 U = \begin{pmatrix} -1 & 0 & 0 \\ 0 & 0 & - \frac{1}{\sqrt{2}} \\ 0 & 0 & 0 \end{pmatrix} \text{ and } U T_2 U = \begin{pmatrix} 0 & 0 & 0 \\ - \frac{1}{2 \sqrt{2}} & - \frac{1}{2} & - \frac{1}{2\sqrt{2}} \\ \frac{1}{2} & - \frac{1}{\sqrt{2}} & \frac{1}{2} \end{pmatrix}.
\]
We note that $\fr(T^t)$ has eigenvalue $1$ of multiplicity $1$. The corresponding eigenvector is $e_1 + e_3$. The above calculation shows that only the semi-simple part of $T^t$ is in the closed ball. The similarity orbit of $T^t$ itself never intersects the closed ball.
\end{example}

\bibliographystyle{abbrv}
\bibliography{Bibs/cuntz}

\end{document}